\documentclass[11pt]{article}
\usepackage{mathtools}
\usepackage{amssymb}
\usepackage{amsthm}
\usepackage{bigints}
\usepackage{mathdots}
\usepackage[]{graphicx} 
\usepackage{wrapfig}
\usepackage{float}
\usepackage{verbatim}
\usepackage[margin=1in]{geometry}
\usepackage{url}
\usepackage[utf8]{inputenc}
\usepackage[T1]{fontenc}
\newtheorem{proposition}{Proposition}[section]
\newtheorem{theorem}{Theorem}[section]
\newtheorem{lemma}{Lemma}[section]
\newtheorem*{remark}{Remark}
\title{The SQG Equation as a Geodesic Equation}
\author{Pearce Washabaugh}

\begin{document}

\maketitle

\begin{abstract}

We demonstrate that the surface quasi-geostrophic (SQG) equation given by
$$\theta_t + \left<u, \nabla \theta\right>= 0,\;\;\; \theta = \nabla \times (-\Delta)^{-1/2} u,$$
 is the geodesic equation on the group of volume-preserving diffeomorphisms of a Riemannian manifold $M$ in the right-invariant $\dot{H}^{-1/2}$ metric. We show by example, that the Riemannian exponential map is smooth and non-Fredholm, and that the sectional curvature at the identity is unbounded of both signs.

\end{abstract}

\section{Introduction}
As discussed by Choi et al. \cite{Choi14}, there is a large number of model equations of the full 3D Euler equations  that have been investigated analytically. Some of these equations arise naturally as geodesic equations of right-invariant metrics of diffeomorphism and volume-preserving diffeomorphism (volumorphism) groups. For example, a special case of the generalized Constantin-Lax-Majda model first discussed by Okamoto et al. \cite{Okamoto08} is the Wunsch equation \cite{Wunsch10}, 

\begin{equation}\label{Wunsch}
\omega_t  + u\omega_x +2\omega u_x=0,\;\;\; \omega = Hu_x,
\end{equation}

which is the geodesic equation on the diffeomorphism group of the circle in the $\dot{H}^{1/2}$ right-invariant metric. 
For equations arising in such a fashion, it is then natural to investigate their associated geometric properties in the manner initiated by V. Arnold \cite{Arnold66}. In this paper, we demonstrate that the well known surface quasi-geostrophic (SQG) equation is the geodesic equation on the volumorphism group of a 2D manifold in the $\dot{H}^{-1/2}$ inner product. The SQG equation on a Riemannian manifold $M$ with metric $\left<\cdot ,\cdot \right>$ is given by

\begin{equation}\label{SQG}
\theta_t + \left<u, \nabla \theta\right>= 0,\;\;\; u = \mathcal{R}^\perp \theta,
\end{equation}

where $ \mathcal{R}^\perp$ is the perpendicular Riesz transform. Many of the basic mathematical properties of this equation were first investigated by Constantin-Majda-Tabak \cite{Constantin94}. Importantly, while this equation is known to have solutions for short time, the global in time existence problem is still open. It is believed by some (see e.g. Constantin-Majda-Tabak \cite{Constantin94}) that the blow-up mechanism (should it exist) of this equation may have very similar properties to that of the full 3D Euler equations. As Bauer-Kolev-Preston \cite{Bauer15} did for the Wunsch equation \eqref{Wunsch}, in this paper we investigate some of the basic geometric properties of the SQG equation \eqref{SQG}. We perform the necessary computations in a variety of domains in order to keep the paper as simple as possible. Looking forward, it will be necessary to firmly establish the theory of this equation in a single domain (as in Escher-Kolev \cite{Escher14} for positive fractional order Sobolev metrics on the diffeomorphism group of the circle).\\
\\The following is a list of the geometric properties associated to the SQG equation we explore:

\begin{itemize}
\item \emph{Smoothness of the Riemannian exponential map}\\

The Riemannian exponential map on the volumorphism group in a Riemannian metric takes a velocity field (tangent vector) to the solution of the geodesic equation of the metric at time one. In our case, the geodesic equation is equivalent to SQG \eqref{SQG} and the geodesic evaluated at time one is a particle trajectory map. We may then ask whether or not this map is smooth. This question is partially answered by Constantin-Vicol-Wu \cite{Constantin14} where the authors demonstrate analyticity of the particle trajectories. Here, we are also concerned with smooth dependence on the initial data. In this paper we demonstrate that the Lagrangian formulation of SQG has smooth dependence on the initial data in the case that the underlying manifold is $\mathbb{R}^2$. This suggests in general that the Riemannian exponential map will be smooth for the $\dot{H}^{-1/2}$ right invariant metric on the volumorphism group of any manifold.\\
 
 \item \emph{Non-Fredholmness of the Riemannian exponential map}\\
 
  Next, we show that the Riemannian exponential map on $\mathcal{D}_\mu(\mathbb{S}^2)$ in the $\dot{H}^{-1/2}$ inner product is not a Fredholm map in the sense of Smale \cite{Smale65}.  Ebin et al. \cite{Ebin06} showed that, for $M$ a compact 2D Riemannian manifold without boundary, in the $L^2$ metric on $\mathcal{D}_\mu(M)$, the exponential map is a nonlinear Fredholm map of index zero. It was also demonstrated that the exponential map is \emph{not} Fredholm in the 3D situation. This points to a significant difference between 2D and 3D hydrodynamics. Fredholmness has been used to obtain results about the $L^2$ geometry of the 2D volumorphism group, such as an infinite dimensional version of the Morse Index Theorem (see Misio\l{}ek and Preston \cite{Misiolek09}), and a version of the Morse-Littauer Theorem (see Misio\l{}ek \cite{Misiolek15}). In this paper, we solve the Jacobi equation along a simple rotational flow to demonstrate the existence of an epiconjugate point that is not monoconjugate (see Grossman \cite{Grossman65}); thus the exponential map is non-Fredholm. Preston \cite{Preston10} showed that there is a concrete connection between blow up and the existence of conjugate points, thus our argument here provides evidence that the blow up behavior of 2D SQG is similar geometrically to that of 3D Euler.\\
  
 \item \emph{The sectional curvature of the volumorphism group in the $\dot{H}^{-1/2}$ metric and vanishing geodesic distance}\\
 
 Finally, as was suggested by Arnold \cite{Arnold66}, the sectional curvature of the volumorphism group is helpful in predicting Lagrangian stability in fluid flows. While intuitively appealing, little is currently known about this relationship. One would like to be able to use the Rauch Comparison test; however to make use of this theorem, one must bound the sectional curvature with either a strictly positive or strictly negative constant. For the $L^2$ metric, various partial situations were investigated by Preston \cite{Preston02} in which it was demonstrated that the situation is quite complicated if one does not have these bounds. In this paper we demonstrate that $\mathcal{D}_\mu(\mathbb{T}^2)$ (the volumorphism group of the flat torus) in the $\dot{H}^{-1/2}$ metric exhibits arbitrarily large curvature of both signs. As was first conjectured by Michor and Mumford \cite{Michor05}, we conjecture that the unbounded curvature implies that the  geodesic distance on this space vanishes.
 
\end{itemize}

\subsection*{Acknowledgments}
I would like to thank my advisor, Dr. Stephen Preston, for the original idea of this paper, as well as his constant advice and encouragement.

\section{The SQG equation as a geodesic equation}

Tao \cite{Tao14} demonstrated that solutions to the SQG equation are the critical points of a functional obtained from the inertia operator $A = (-\Delta)^{-1/2}$. Assume that $M$ is a 2D Riemannian manifold, possibly with boundary. It is known from Arnold \cite{Arnold66} that, in the case that $M$ is compact $\mathcal{D}_{\mu,ex}(M)$, the group of exact volumorphisms, can be thought of as an infinite-dimensional Lie-group. In this section, we demonstrate that the SQG equation is the geodesic equation on $\mathcal{D}_{\mu,ex}(M)$ in the $\dot{H}^{-1/2}$ metric, obtained from the inertia operator $A$ (we will be working on this space formally in the case that $M$ is not compact). However, in comparison with Bauer-Kolev-Preston \cite{Bauer15} we consider the $\dot{H}^{1/2}$ metric on $C^\infty(M)$. In other words, for $\phi,\psi\in C^\infty(M)$ such that $\phi|_{\partial M},\psi|_{\partial M} =0 $,

\begin{equation}\label{H1/2}
\left<\left<\phi,\psi\right>\right>_{\dot{H}^{1/2}} = \int_M (\Delta^{1/2} \phi) \psi d\mu.
\end{equation}

We can make this into a metric on $T_{id}\mathcal{D}_{\mu,ex}(M)$ by letting $u = \nabla^\perp \phi$ and $v = \nabla^\perp \psi$ which gives us:

\begin{equation}\label{H-1/2}
\left<\left<u,v\right>\right>_{\dot{H}^{-1/2}} = \int \left< (-\Delta)^{-1/2}u,v\right> d\mu.
\end{equation}

In other words, a possible analogy to the 1D $\dot{H}^{1/2}$ metric is the $\dot{H}^{-1/2}$ metric on $\mathcal{D}_{\mu,ex}(M)$. Using push-forward by right translation we then obtain a right invariant metric on all of $\mathcal{D}_{\mu,ex}(M)$. Now, on $C^\infty(M)$ the Euler-Arnold equation is given by

\begin{equation}\label{EulerArnoldStream}
\psi_t= -\mbox{ad}^*_\psi \psi,
\end{equation}

where $\mbox{ad}^*_\psi:\mathfrak{g}\rightarrow\mathfrak{g}$ is given by 

$$\left<\left<\mbox{ad}^*_\psi \phi,\nu \right>\right>_{\dot{H}^{1/2}} =\left< \left<\phi,\mbox{ad}_\psi\nu\right>\right>_{\dot{H}^{1/2}},$$

and $\mbox{ad}_\psi \nu = -\{\psi,\nu\} = \{\nu,\psi\}$ is the negative of the Poisson bracket, which for our purposes will be given by

$$\{\nu,\psi\} = d\nu(\nabla^\perp \psi) = \left<\nabla \nu , \nabla^\perp \psi\right>.$$

\begin{remark} 
Our $\mbox{ad}^*$ operator here defined on the Lie algebra $\mathfrak{g}$ is the same as that used in \cite{Misiolek09}, which is slightly different from the usual $\mbox{ad}^*$ operator defined on the dual Lie algebra $\mathfrak{g}^*$. 
\end{remark}

\begin{theorem}

The SQG equation is the geodesic equation on $\mbox{D}_{\mu,ex}(M)$ and on $\mbox{D}_{\mu}(M)$ in the $\dot{H}^{-1/2}$ metric.

\end{theorem}

\begin{proof}

For $\mathcal{D}_{\mu,ex}(M)$, we compute $\mbox{ad}^*$ on $C^\infty(M)$. For $\psi,\phi, \nu \in C^\infty(M)$ such that $\psi|_{\partial M},\phi|_{\partial M}, \nu|_{\partial M} =0$,

$$\left<\left< \mbox{ad}^*_\psi \phi,\nu\right>\right> = \left< \left<\phi,\mbox{ad}_\psi\nu\right>\right> =\left< \left<\phi,\{\nu,\psi\}\right>\right>$$

$$ = \int_M A(\phi) \{\nu,\psi\} d\mu = \int_M A(\phi) \left< \nabla \nu, \nabla^\perp \psi\right> d\mu$$

$$ = -\int_M \nu \; \mbox{div} (A(\phi)\nabla^\perp \psi)d\mu  + \int_{\partial M} \nu \left<A(\phi)\nabla^\perp \psi, n\right>d\tilde{\mu}=  -\int_M \nu \left<\nabla A(\phi), \nabla^\perp \psi\right> d\mu,$$ 

where $n$ is the unit normal to the boundary and $d\tilde{\mu}$ is the boundary measure. Note then, that the boundary term vanishes. Thus

$$\mbox{ad}^*_\psi \phi = A^{-1} \left(-\left<\nabla A(\phi), \nabla^\perp \psi\right> \right),$$

and the geodesic equation \eqref{EulerArnoldStream} becomes

$$A\psi_t =- \left<\nabla A\psi, \nabla^\perp \psi\right>. $$

Letting $A\psi = \theta$ and $u = \nabla^\perp \psi$ we obtain the SQG equation \eqref{SQG}. The case for $\mathcal{D}_\mu(M)$ follows by computing $\mbox{ad}^*$ for vector fields in the $\dot{H}^{-1/2}$ metric and then applying $\nabla \times A$ to both sides of the equation. 
\end{proof}

\begin{remark}
Note that if $M$ admits harmonic vector fields, then this inner product is degenerate on $\mathcal{D}_\mu(M)$. Thus in these situations we are really considering this as a geodesic equation on a homogenous space. One needs to verify that the inertia operator is invariant with respect to $\mbox{Ad}$ as is done in Khesin and Misio\l{}ek \cite{Khesin03}. A short computation shows that this holds.
\end{remark}

\section{Smoothness of the Riemannian Exponential Map}
 The Riemannian exponential map on a Riemannian manifold $\mathbf{N}$ at a point $(p,v)\in T\mathbf{N}$, the tangent bundle of $\mathbf{N}$ is given by
 
 $$ \mbox{exp}: T\mathbf{N} \rightarrow \mathbf{N}$$
 
 \begin{equation}\label{exponentialmap}
 \mbox{exp}_p(v) = \gamma_v(1),
 \end{equation}
 
 where $v\in T_p\mathbf{N}$ and $\gamma_v(1)$ is the geodesic through $p$ with initial velocity $v$ evaluated at time $1$. As in Constantin-Vicol-Wu \cite{Constantin14}, we may write SQG \eqref{SQG} as an ODE on a Banach manifold $\mathbf{M}$ (to be defined below), which will correspond to a sub-manifold of $T\mathbf{N}$. The ODE will look like

\begin{equation}\label{generalode}
\frac{dX}{dt} = F(X,\theta_0).
\end{equation}

Then $p $ and $v$ will correspond to $X$ and $\theta_0$ respectively. Smoothness of the Riemannian exponential map is then equivalent to the above equation having smooth (in time) solutions that vary smoothly with respect to the initial data. There are results establishing smoothness of exponential maps in general Sobolev metrics. For example, Escher-Kolev \cite{Escher14} did this for Sobolev metrics of order $s>\frac{1}{2}$ on the diffeomorphism group of the circle. However, our Sobolev metric is of \emph{negative} index, thus no known results apply. Constantin-Vicol-Wu \cite{Constantin14} demonstrated that the individual particle paths are analytic as functions of time. They proved the following theorem that we cite here for convenience:

\begin{theorem}[Constantin-Vicol-Wu \cite{Constantin14}]

Consider initial data $\theta_0\in C^{1,\gamma}\cap W^{1,1}$, and let $\theta$ be the unique maximal solution to \eqref{SQG}, with $\theta \in L_{loc}^\infty([0,T_*);C^{1,\gamma}\cap W^{1,1})$. Given any $t\in [0,T_*)$, there exists $T\in (0,T_*-t)$, with $T = T(||\nabla u||_{L^\infty(t,(t+T_*)/2;L^\infty)})$, and $R>0$ with $R = R(t,||\theta_0||_{C^{1,\gamma}\cap W^{1,1}},\gamma)$, such that

\begin{equation}\label{Xestimate}
||\partial_t^n(X-id)||_{L^\infty(t,t+T;C^{1,\gamma})}\leq C n! R^{-n}
\end{equation}

holds for any $n\geq 0 $. Here $C$ is a universal constant, and the norm $||X-id||_{C^{1,\gamma}}$ is defined in equation \eqref{C1gammanorm}. In particular, the Lagrangian trajectory $X$ is a real analytic function of time, with radius of analyticity $R$.

\end{theorem}

Our purpose here is to demonstrate that the Riemannian exponential map is smooth, which is equivalent to demonstrating smooth dependence on on the initial data ($X$ and $\theta_0$) in \eqref{generalode}. We will do so not on $T\mathbf{N}$ but instead on a closely related Banach affine space, denoted by $\mathbf{M}$ to be defined below. We will show how the argument from Constantin-Vicol-Wu \cite{Constantin14} can be extended to obtain smoothness of $F$ so that the following theorem from Lang (Chapter 4, Theorem 1.11 in \cite{Lang95}) can be applied:

\begin{theorem}[Lang \cite{Lang95}]\label{Langtheorem}

Let $J$ be an open interval in $\mathbb{R}$ containing $0$ and $U$ open in the Banach space $\mathbf{E}$. Let

$$f :J\times U \rightarrow \mathbf{E}$$

be a $C^p$ map with $p\geq 1$, and let $x_0\in U$. There exists a unique local flow for $f$ at $x_0$. We can select an open subinterval $J_0$ if $J$ containing $0$ and an open subset $U_0$ of $U$ containing $x_0$ such that the unique local flow 

$$\eta: J_0\times U_0 \rightarrow U$$

is of class $C^p$, and such that $D_2\eta$ satisfies the differential equation

$$D_1D_2\eta(t,x) = D_2f(t,\eta(t,x))D_2\eta(t,x)$$

on $J_0\times U_0$ with initial condition $D_2\eta(0,x) = id$.

\end{theorem}

Note that in our case, $f = F$ will be autonomous, so we will only need smoothness in $U$. This in turn demonstrates smoothness of the Riemannian exponential map in this situation. In fact, most of the argument has been done in Constantin-Vicol-Wu \cite{Constantin14}. Here we are extending their argument to obtain smooth dependence of the initial data and hence smoothness of the Riemannian exponential map. 

\begin{remark}

As Constantin-Vicol-Wu \cite{Constantin14} demonstrated analyticity of the Lagrangian trajectories, one can quite likely extend these arguments to obtain analyticity of the Riemannian exponential map as was done by Shnirelman \cite{Shnirelman12} for the $L^2$ metric on $\mathcal{D}_\mu^sT^3$, the group of order $s$ Sobolev class volumorphisms for $s>5/2$.

\end{remark}

\subsection{The Domain $\mathbf{M}$}
Following the strategy of Chapter 4 of Majda-Bertozzi \cite{Majda02} we enlarge the space of volume preserving maps to allow for maps with some compressibility. This allows us to apply Theorem \ref{Langtheorem} directly as we can deal with an open subset of a Banach space rather than a submanifold. Constantin-Vicol-Wu \cite{Constantin14} analyze SQG explicitly on the volume preserving case, however, as they mention, their argument extends to the compressible case. We let 

$$\mathbf{N} =  C^{1,\gamma}(\mathbb{R}^2,\mathbb{R}^2),$$

for 
$$C_b^{1,\gamma}(\mathbb{R}^2,\mathbb{R}^2) = \{Y:\mathbb{R}^2\rightarrow \mathbb{R}^2 : ||Y||_{1,\gamma}<\infty\},$$

$$C^{1,\gamma}(\mathbb{R}^2,\mathbb{R}^2) = \{id+C_b^{1,\gamma}(\mathbb{R}^2,\mathbb{R}^2) \},$$

where 

\begin{equation}\label{C1gammanorm}
||Y||_{1,\gamma} = ||Y||_{L^\infty} + ||\nabla Y  ||_{L^\infty} + \lbrack \nabla Y \rbrack_{C^\gamma}.
\end{equation}

Here, the $L^\infty$ norm is taken as the largest absolute value of an entry of the corresponding vector or matrix. Note that $C^{1,\gamma}(\mathbb{R}^2,\mathbb{R}^2)$ is an affine Banach space. Then we may identify 

$$T\mathbf{N} = C^{1,\gamma}(\mathbb{R}^2,\mathbb{R}^2) \times C_b^{1,\gamma}(\mathbb{R}^2,\mathbb{R}^2).$$

Points in $T\mathbf{N}$ are of the form $(X,u)$. We define the domain on which we'll be solving SQG to be

$$\mathbf{M} = C^{1,\gamma}(\mathbb{R}^2,\mathbb{R}^2)\times\left(C^{1,\gamma}(\mathbb{R}^2)\cap W^{1,1}(\mathbb{R}^2)\right),$$

 where $C^{1,\gamma}(\mathbb{R}^2)$ is the space of H\"{o}lder continuous functions on $\mathbb{R}^2$ and $W^{1,1}(\mathbb{R}^2)$ is the corresponding Sobolev space of functions on $\mathbb{R}^2$. We note then that the perpendicular Riesz transform, 

$$\mathcal{R}^\perp: C^{1,\gamma}(\mathbb{R}^2) \cap W^{1,1}(\mathbb{R}^2)\rightarrow  C_b^{1,\gamma}(\mathbb{R}^2,\mathbb{R}^2),$$

$$\mathcal{R}^\perp: \theta \mapsto u,$$

gives a correspondence between $\mathbf{M}$ and a subset of $T\mathbf{N}$. We must also select the open set $\mathcal{U}\subset \mathbf{M}$ on which we'll define $F$, as in the theorem from Lang. As discussed above, our problem is that we would like to focus only on $X$ such that $\mbox{det} \nabla_aX(a)= 1$, but this does not yield us an open subset of $\mathbf{M}$, thus as in chapter 4 of Majda-Bertozzi \cite{Majda02}, we enlarge our domain to include some compressibility. We define 

$$\mathcal{O} = \{  X = id+Y\in C^{1,\gamma}(\mathbb{R}^2,\mathbb{R}^2):\frac{9}{10}< \underset{a\in \mathbb{R}^2}{\inf} \mbox{det }\nabla_a X(a) ,\;\;||Y||_{1,\gamma}<c\},$$

$$\mathcal{U} = \mathcal{O}\times \left(C^{1,\gamma}(\mathbb{R}^2)\cap W^{1,1}(\mathbb{R}^2)\right),$$

where

$$c=\frac{7}{20}.$$

That $\mathcal{O}$ and hence $\mathcal{U}$ are open in their respective spaces follows from continuity of $\underset{a\in \mathbb{R}^2}{\inf} \mbox{det}$ and $||\cdot ||_{1,\gamma}$. We will also require more properties of $\mathcal{O}$ that are used by Constantin \cite{Constantin14}. The main fact they use to obtain analyticity of the particle trajectories is the following chord-arc condition, which is satisfied by solutions of SQG with Lipschitz velocity field $u$. That is, there is a constant $\lambda$ such that 

\begin{equation}\label{chordarc}
\lambda^{-1} \leq \frac{|a-b|}{|X(a,t)- X(b,t)|} \leq \lambda.
\end{equation}

We claim that there exists such a constant $\lambda$ for $ X \in \mathcal{O}$. First we need the following result:

\begin{lemma}\label{Xhomeo}

Suppose that $X=id+Y\in \mathcal{O}$, then $X$ is a homeomorphism of $\mathbb{R}^2$ onto $\mathbb{R}^2$.

\end{lemma}

\begin{proof}
Since $X\in \mathcal{O}$, $||\nabla Y||_{L^\infty}<c$. Hence the largest an entry of $\nabla X$ can be in magnitude is $c+1$. Writing out the inverse of $\nabla X$ explicitly, combined with the fact that 

$$ \frac{9}{10}< \underset{a\in \mathbb{R}^2}{\inf} \mbox{det }\nabla_a X(a) $$

yields:

$$||\nabla X^{-1}||_{L^\infty}<\frac{3}{2}.$$

As is similarly discussed in Majda-Bertozzi \cite{Majda02}, a result of Hadamard (\cite{Berger77}, pg. 222) demonstrates that if $X\in \mathcal{O}$ and there exists a constant $d$ such that

$$||\nabla X^{-1}||_{L^\infty} \leq d,$$

then $X$ is a homeomorphism of $\mathbb{R}^2$ onto $\mathbb{R}^2$.
\end{proof}

In order to make use of the estimates of Constantin-Vicol-Wu \cite{Constantin14}, it is necessary that $\lambda \in (1,\frac{3}{2}\rbrack$ in \eqref{chordarc}.

\begin{lemma}

If $X\in \mathcal{O}$, then $X$ satisfies the chord-arc condition \eqref{chordarc} for $\lambda =\frac{3}{2}$. 

\end{lemma}

\begin{proof}

Since $X\in C^{1,\gamma}(\mathbb{R}^2,\mathbb{R}^2)$, given $a,b \in \mathbb{R}^2$, 

$$|X(a) - X(b)| \leq |\nabla X|_{L^\infty} |a-b|.$$

Thus

\begin{equation}\label{lipschitzineq}
\frac{1}{|\nabla X|_{L^\infty} }\leq \frac{|a-b|}{ |X(a) - X(b)| }.
\end{equation}

Hence,

$$|X^{-1}(\alpha) - X^{-1}(\beta)| \leq ||\nabla X^{-1}||_{L^\infty} |\alpha-\beta| < \frac{3}{2} |\alpha-\beta|,$$

where we have used the bound on $||\nabla X^{-1}||_{L^\infty}$ obtained in the proof of lemma \ref{Xhomeo}. Choosing $\alpha = X(a)$ and $\beta = X(b)$ yields:

$$\frac{|a-b|}{ |X(a) - X(b)| }< \frac{3}{2}.$$

This combined with \eqref{lipschitzineq} gives us the claim.

\end{proof}

For the SQG equation, we recover the velocity field from the vorticity by

$$u(x) = \mathcal{R}^\perp \theta(x) = \int_{\mathbb{R}^2} \frac{(x-y)^\perp}{2\pi |x-y|^3}\theta(y)dy = \int_{\mathbb{R}^2} K(x-y) \theta(y)dy,$$

where all integrals are considered in the principal value sense. The SQG equation itself says that

$$\theta(X(b,t),t) = \theta_0(b).$$

Hence from the flow equation we obtain

\begin{equation}\label{Fnovarsub}
\frac{dX}{dt} (a,t) = \int_{\mathbb{R}^2} K(X(a,t)-y)\theta(y,t)dy.
\end{equation}

Then the precise system of ODEs we will be studying is given by

$$\frac{dX}{dt} = F(X,\theta_0),$$

where 

\begin{equation}\label{F}
F (X,\theta_0)(a) =  \int_{\mathbb{R}^2} \frac{(X(a) - X(b))^\perp}{2\pi |X(a) - X(b)|^3}\theta_0(b)J_X(b)db
\end{equation}

$$ = \int_{\mathbb{R}^2} K(X(a) - X(b))\theta_0(b)J_X(b)db,$$

is the Riesz transform of $\theta_0$ when $X = \mbox{id}$ and $J_X(b)= \mbox{det} \nabla_bX(b)$. We also wish to obtain $\nabla_aF(X,\theta_0)(a)$. This follows in essentially the same manner as what is done by Constantin et al. \cite{Constantin14},

\begin{equation}\label{gradF}
\nabla_a F(X,\theta_0)= \nabla_a X(a,t)\int_{\mathbb{R}^2} K(X(a,t)-X(b,t))(\nabla_b^\perp X^\perp)(b,t)(\nabla_b\theta_0)(b)J_X(b)db.
\end{equation}

The fact that $F$ is well defined is nontrivial, but follows from the smoothness argument. We will also need the following estimates on compositions of functions in $C^{1,\gamma}(\mathbb{R}^2,\mathbb{R}^2)$. 

\begin{lemma}\label{complemma}

Let $X\in \mathcal{O}$, $Z\in C^{1,\gamma}(\mathbb{R}^2,\mathbb{R}^2)$. We have

\begin{equation}\label{compestimate}
||Z\circ X-id||_{1,\gamma} \leq C_1(1+ ||Z-id||_{1,\gamma}),
\end{equation}

where $C_1$ is determined entirely by $\mathcal{O}$.

\end{lemma}

\begin{proof}

By definition,

$$||Z\circ X-id||_{1,\gamma} = ||Z\circ X-id||_\infty + ||\nabla (Z\circ X)-I||_\infty + |\nabla (Z\circ X) |_\gamma,$$ 

where $I$ is the identity matrix. There exists $W\in C^{1,\gamma}(\mathbb{R}^2)\times C^{1,\gamma}(\mathbb{R}^2)$ such that $Z = id+W$. Hence for the first term above,

\begin{equation}\label{firstcompterm}
||Z\circ X-id||_\infty = ||X + W\circ X-id||_\infty \leq c + ||Z-id||_\infty,
\end{equation}

where we recall that $c =\frac{7}{20}$. For the second term,

\begin{equation}\label{secondcompterm}
||\nabla_a(Z\circ X)(a)-I||_\infty = ||\nabla_a(X+W\circ X)(a)-I||_\infty \leq c +  ||\nabla_a(W\circ X)(a)||_\infty\leq C(1+||\nabla Z-I||_\infty)
\end{equation}

for some constant C. Similarly, for the third term,

\begin{equation}\label{thirdcompterm}
|\nabla_a(Z\circ X)(a)|_\gamma \leq C(1+|\nabla_aZ(a)|_\gamma).
\end{equation}

Combining \eqref{firstcompterm}, \eqref{secondcompterm}, and \eqref{thirdcompterm} gives us the claim.

\end{proof}

We will also need a sense of how $F$ behaves under composition of functions.

\begin{lemma}\label{Fcomp}

Let $Y\in \mathcal{O}$. Then 

$$F(X,\theta_0) \circ Y = F(X\circ Y,\theta_0\circ Y).$$

\end{lemma}

\begin{proof}

$$F(X,\theta_0) \circ Y = \int_{\mathbb{R}^2} \frac{(X(Y(a))-X(b))^\perp}{2\pi|X(Y(a))-X(b)|^3}\theta_0(b)J_X(b)db.$$

Let $b = Y(s)$, for $b,s\in \mathbb{R}^2$. Then

$$=  \int_{\mathbb{R}^2} \frac{(X(Y(a))-X(Y(s)))^\perp}{2\pi|X(Y(a))-X(Y(s))|^3}\theta_0(Y(s))J_X(Y(s))J_Y(s)ds =  F(X\circ Y,\theta_0\circ Y). $$

\end{proof}

\subsection{Smoothness of the ODE}
Since $\mathcal{U}$ is an open subset of the affine space $\mathbf{M}$, we can demonstrate smoothness of $F$ by showing that the operator norms of its partial derivatives, $d_{X}^n(F)$ and $d_{\theta_0}^n(F)$, are bounded in some uniform way on $\mathcal{U}$. As discussed above, we show that $F$ is smooth by adapting the argument made by Constantin et al. \cite{Constantin14}. The theorem of Lang can be used after the following theorem:

\begin{theorem}

$F$ is infinitely Fr\'{e}chet differentiable on $\mathcal{U}$.

\end{theorem}

\begin{proof}

The idea is the following, we wish to obtain a bound on $||d_{X}^nF||_{\mathcal{L}^n}$ (where $\mathcal{L}^n$ is the corresponding space of multilinear maps):

$$||d_{X}^nF(X_1,...,X_n)||_{1,\gamma} \leq ||d_{X}^nF||_{\mathcal{L}^n} \cdot ||X_1||_{1,\gamma}\cdots ||X_1||_{1,\gamma} $$
$$=  ||d_{X}^nF||_{\mathcal{L}^n}, $$

for all $X_i \in \partial B_1(0) \subset \mathbf{E}_1$ where $||d_{X}^nF||_{\mathcal{L}^n}$ is independent of $X$. Now, if $X(t)$ is a solution to SQG with initial condition $\theta_0$ and with $J_X = 1$, Constantin et al. \cite{Constantin14} estimates (for our purposes) for $n\geq0$:

\begin{equation}\label{mainbound}
||\left.\partial_t^{n+1}\right|_{t=0}X(t) ||_{1,\gamma}= ||\left.\partial_t^n\right|_{t=0} F(X(t), \theta_0) ||_{1,\gamma}\leq C n! R^{-n},
\end{equation}

where $R = R(||\theta_0||_{C^{1,\gamma}\cap W^{1,1}},\gamma,\lambda)>0$ and $C$ is another constant. This gives us that the incompressible particle trajectories are analytic in time. The first point is that, as Constantin et al. \cite{Constantin14} mentions, this argument can be extended for $J_X\neq 1$. The terms then involve a Jacobian and its time derivatives, but these are bounded, hence the same estimates go through but with modified constants. We now show that this bound also gives us that $F$ is smooth in its $X$ component, i.e. it provides the desired bound on $||d_{X}^nF||_{\mathcal{L}^n}$. In the case $n=0$, bound \eqref{mainbound} gives us that $F$ is well defined at the identity, i.e. 

$$|| F(id, \theta_0) ||_{1,\gamma}\leq CR^{-1}. $$

Away from the identity, we make use of lemmas \ref{complemma} and \ref{Fcomp} to obtain the desired bound. We will now proceed by induction. The idea is the following: using the multivariate Fa\'{a} di Bruno formula we can expand $\partial_t^n F(X(t), \theta_0)|_{t=0} $, note in particular that the last term is $d_{X}^nF(X_1,...,X_1)$ where $X_1 = \left.\partial_t\right|_{t=0}X(t)$. Assuming that the bound holds in the case $n-1$, we can subtract out bounded lower order terms from $\partial_t^n F(X(t), \theta_0)|_{t=0} $ to obtain that $d_{X}^nF(X_1,...,X_1)$ is bounded. One can then obtain a bound on the full operator $d_{X}^nF(X_1,...,X_n)$ by polarization. This gives us smoothness at $X = id$. We will then use lemmas \ref{complemma} and \ref{Fcomp} to obtain smoothness for any $Y\in \mathcal{O}$. Here we do this explicitly for the case $n=2$. Suppose that $X(t)$ is a smooth curve in $\mathcal{O}$ such that $X(0)= X$ and $\left.\partial_t\right|_{t=0}X(t) = X_1$, with $||X_1||_{1,\gamma} = 1$, and such that $X$ is a solution to SQG. We have, 

$$\left.\partial_t^{3}\right|_{t=0}X(t) = \left.\partial_t^{2}\right|_{t=0} F(X(t),\theta_0) $$

$$ = (d^2F)_X(X_1,X_1) + (dF)_X(\tilde{X}_2),$$

where $\tilde{X}_j = \left.\partial_t^j\right|_{t=0}X(t)$, hence we may write

$$(d^2F)_X(X_1,X_1) = \left.\partial_t^{3}\right|_{t=0}X(t)-(dF)_X(\tilde{X}_2).$$

If $X = id$, then by \eqref{mainbound} and the inductive hypothesis we have that 

$$ ||(d^2F)_X(X_1,X_1)||_{1,\gamma} \leq C(R^{-1}+ 2R^{-2}).$$

Now, to obtain $(d^2F)_X(X_1,X_2)$ for any other $X_2$ we use polarization to obtain that

$$||(d^2F)_X(X_1,X_2)||_{1,\gamma} = \frac{1}{2}||(d^2F)_X(X_1+X_2,X_1+X_2)-dF_X(X_1,X_1)- dF_X(X_2,X_2)||_{1,\gamma}$$
$$\leq 2C(R^{-1}+ 2R^{-2}).$$

Now, if $Z\in \mathcal{O}$, one can verify, in a manner similar to lemma \ref{Fcomp}, that 

$$\left(\partial_t^2|_{t=0} F(X(t),\theta_0)\right)\circ Z= \left(\partial_t^2|_{t=0} F(X(t),\theta_0)\circ Z\right).$$

This then gives us that:

$$(d^2F)_{X\circ Z}(X_1\circ Z,X_2\circ Z) = d^2F_X(X_1,X_2)\circ Z.$$

By lemma \ref{complemma} we obtain the desired bound.\\
\\
Finally, we note that if $||\theta_0||_{C^{1,\gamma}\cap W^{1,1}}= 1$, then we have a bound $F(X,\theta_0) \leq C$. Since $F$ is linear in $\theta_0$, this gives us that $F$ is a bounded linear operator in $\theta_0$. Hence $F$ is smooth in $\theta_0$ and the Riemannian exponential map is smooth by Lang's Theorem \ref{Langtheorem}.

\end{proof}

\section{Non-Fredholmness of the Riemannian Exponential Map}
Preston  \cite{Preston02} and Rouchon \cite{Rouchon91} demonstrated that just as the geodesic equation on the volumorphism group in the $L^2$ metric splits in the Lie algebra, so does the Jacobi equation. Here, we begin by citing the following proposition from Misio\l{}ek-Preston \cite{Misiolek09}:

\begin{proposition}[Misio\l{}ek-Preston \cite{Misiolek09}]

Suppose $G$ is any Lie group with a (possibly weak) right-invariant metric. Let $\eta(t)$ be a smooth geodesic with $\eta(0) =e$ and $\dot{\eta}(0)=u_0$. Then, every proper Jacobi field $J(t)$ (such that $J(0)=0$) along $\eta$ satisfies the following system of equations on $T_eG$:

\begin{align}
\frac{dY}{dt}-\mbox{ad}_XY=Z \label{linearizedflow}\\
\frac{dZ}{dt}+\mbox{ad}_X^*Z+\mbox{ad}_Z^*X=0, \label{linearizedeuler}
\end{align}

where $J(t) = dR_{\eta(t)}Y(t)$, $\dot{\eta}(t) = dR_{\eta(t)}X(t)$, $Y(0)=0$, and $Z(0)=0$. 
\end{proposition}

The first equation is the linearized flow equation, while the second is the linearized Euler equation. Here we find some explicit solutions to give us non-Fredholmness. We now let $\mathbb{S}^2$ denote the standard two-sphere. In this section we demonstrate the following:

\begin{theorem}

The Riemannian exponential map on $\mathcal{D}_\mu(\mathbb{S}^2)$ in the $\dot{H}^{-1/2}$ inner product is non-Fredholm.  

\end{theorem} 

\begin{proof}
We will let $X = \nabla^\perp f$, $Y= \nabla^\perp g$, and $Z = \nabla^\perp h$. Note then that  \eqref{linearizedflow} and \eqref{linearizedeuler} give us:
\begin{equation}\label{linflowstream}
g_t + \{f,g\} = h,
\end{equation}

\begin{equation}\label{lineulerstream}
\psi_t + \frac{1}{\sin \phi}\left(f_\phi \psi_r-f_r \psi_\phi\right)+\frac{1}{\sin \phi}\left(h_\phi \theta_r-h_r \theta_\phi \right) = 0,
\end{equation}

respectively, where $\psi = \sqrt{-\Delta}(h)$ . Here $f$ (and hence $\theta$) will be determined by a solution to SQG. We must impose the condition that $g(0) = 0$, so that we have a proper Jacobi field. Let $f = -\cos \phi$. Then, since $\Delta f = 2 \cos \phi$ we have that

$$A(\cos(\phi)) = \sqrt{-\Delta}(\cos \phi) = \sqrt{2} \cos \phi.$$

We note that $X$ and $\theta = \nabla \times A(X)$ solve the SQG equation in spherical coordinates. Then, we let

$$h = \sum h_{nm}(t)\xi_{nm}(\phi)e^{imr}$$

where $\xi_{nm}(\phi)e^{imr}$ is an eigenfunction of $\Delta$:

$$\Delta \xi_{nm}(\phi)e^{imr} = -\lambda_{n}^2\xi_{nm}(\phi)e^{imr},$$

and $\lambda_{n} = \sqrt{n(n+1)}$ with $-n\leq m \leq n$. The solution to \eqref{lineulerstream} is

$$h_{nm}(t) = C_{nm}\cdot \mbox{Exp}\lbrack\frac{i(\sqrt{2}-\lambda_{n})}{\lambda_{n}}mt\rbrack.$$ 

Solving \eqref{linflowstream} for $g$ we obtain 

$$g_{nm}(t) =\frac{-iC_{nm}}{(1+a_{n})m}e^{-int}\left(e^{i(1+a_{n})mt}-1\right),$$

where 

$$a_{n} = \frac{(\sqrt{2}-\lambda_{n})}{\lambda_{n}}.$$

$g_{nm}(t)$ will be zero at 

$$t_{nm}=\frac{2\pi}{(a_{n}+1)m}=\frac{2\pi\sqrt{n(n+1)}}{\sqrt{2}m}.$$

This gives us that 

$$\lim\limits_{n\rightarrow \infty} t_{nn} = \frac{2\pi}{\sqrt{2}}.$$

Thus we have a clustering of conjugate points at $t = \frac{2\pi}{\sqrt{2}}$. So this is a point that is epiconjugate, but not monoconjugate, hence the map is not Fredholm.

\end{proof}

\section{The Sign and Magnitude of the Sectional Curvature}
For a Lie group $G$ with right invariant metric $\left<\left<\cdot ,\cdot \right>\right>$ the non-normalized sectional curvature at the identity in directions $u$ and $v$ is given (as in Arnold \cite{Arnold66}) by

$$\overline{K}(u,v)=\left< \left< R(u,v)v,u\right>\right> = \frac{1}{4}||\mbox{ad}^*_vu + \mbox{ad}^*_uv||^2 - \left< \left<\mbox{ad}^*_uu,\mbox{ad}_v^*v\right>\right>$$
$$ -\frac{3}{4} ||\mbox{ad}_uv||^2 + \frac{1}{2}\left<\left<\mbox{ad}_uv,\mbox{ad}_v^*u - \mbox{ad}_u^*v\right>\right>.$$ 

The normalized sectional curvature is given by

\begin{equation}\label{normcurv}
K(u,v) = \frac{\overline{K}(u,v)}{||u||||v|| - \left<\left<u,v\right>\right>^2}.
\end{equation}
Khesin et al. \cite{Khesin13} computed $\overline{K}$ explicitly for homogeneous Sobolev metrics on $\mathcal{D}_{\mu,ex}(\mathbb{T}^2)$ for vector fields of the form $u = \nabla^\perp \cos(jx+ky)$ and $v = \nabla^\perp \cos(lx+my)$. Here, one may consider the Lie algebra to be $C^\infty(\mathbb{T}^2)$. Then, given a metric $\left<\left<,\right>\right>$ on $\mathcal{D}_{\mu,ex}(\mathbb{T}^2)$ we obtain an inner product on $C^\infty(\mathbb{T}^2)$ given by

$$\left<\left<u,v\right>\right> = \left<\left<\right.\right.\nabla^\perp f,\nabla^\perp g\left.\left.\right>\right>=\int_{\mathbb{T}^2}f\Lambda gd\mu = \left<\left<f,g\right>\right>,$$

for some positive definite, symmetric operator $\Lambda$. For our purposes we will have $F(p)=\sqrt{j^2+k^2}$ where $F$ is the symbol of $\Lambda$ and $p=(j,k)$. 

\begin{proposition}[Khesin et al. \cite{Khesin13}]

Suppose $f(x,y) = \cos(jx+ky)$ and $g(x,y) = \cos(lx+my)$ where $j,k,l,m$ are integer multiples of $2\pi$. Set $p = (j,k)$ and $q=(l,m)$, and let $u = \nabla^\perp f$ and $v = \nabla^\perp g$. Then the non-normalized sectional curvature is given by

\begin{equation}\label{nonnormcurv}
\overline{K}(u,v) = \frac{|p\wedge q|^2}{8}\left\{ \frac{1}{4}\left(F(p)-F(q)\right)^2\left(\frac{1}{F(p+q)}+\frac{1}{F(p-q)}\right)\right.
\end{equation}
$$\left.-\frac{3}{4} \left( F(p+q)+F(p-q)\right)+F(p)+F(q) \right\},$$

where $p\wedge q = jm-kl$.
\end{proposition}

Normalizing the above formula to obtain the usual sectional curvature we have the following:

\begin{theorem}

The sectional curvature of $\mathcal{D}_{\mu,ex}(\mathbb{T}^2)$ in the $\dot{H}^{-1/2}$ metric is unbounded of both signs.

\end{theorem}

\begin{proof}

Let $n\in \mathbb{N}$. First we choose $j=m=2\pi n$ and $k=l=0$. Then using \eqref{normcurv} and \eqref{nonnormcurv} we obtain

$$K(\cos(2n\pi x),\cos(2n\pi y))  \approx-15.0\;n^3,$$

which demonstrates that the sectional curvature can be made to be arbitrarily negative for arbitrarily large $n$.\\
\\
Next we choose $j=m=l=n$ and $k=0$. Then

$$K(\cos(2n\pi x),\cos(2n\pi x+2n\pi y)) \approx 4.3\;n^3,$$

which demonstrates that the curvature can be made to be arbitrarily positive for large $n$.
\end{proof}

\section{Conclusion}
Here we've shown how SQG \eqref{SQG} is the geodesic equation on $\mathcal{D}_\mu(M)$ in the $\dot{H}^{-1/2}$ metric. We've also analyzed some of the basic geometric properties of this manifold: its Riemannian exponential map is smooth and not Fredholm, and its curvature is unbounded of both signs. We saw previously how SQG \eqref{SQG} comes about by considering the $\dot{H}^{1/2}$ metric on stream functions. This may imply that the SQG equation has similarities to the Wunsch equation, which is the geodesic equation on $\mathcal{D}(\mathbb{S}^1)$ in the $\dot{H}^{1/2}$ metric, as was shown by Wunsch \cite{Wunsch10}. This is of importance for the study of the SQG equation because the Wunsch equation blows up, as was shown by Bauer-Kolev-Preston \cite{Bauer15}, and blow-up for the SQG equation is a long-standing open problem.   This is also of importance for infinite dimensional geometry as it shows that negative index Sobolev metrics on diffeomorphism groups can give rise to relevant geodesic equations. As the properties of negative index Sobolev metrics are generally not well known, this will be a fruitful area of further study. There are also many things more to do specifically on $\mathcal{D}_\mu(M)$ in the $\dot{H}^{-1/2}$ metric. For example, as above, we have conjectured that the unbounded curvature implies vanishing geodesic distance (as is discussed by Michor-Mumford \cite{Michor05}). We may also ask whether conjugate points can be concretely linked to possible blow up points, as is done by Preston \cite{Preston10}. Here we've demonstrated that SQG \eqref{SQG} has many geometric similarities to other equations for which blow up is known, such as the Wunsch equation, or unknown, such as 3D Euler. For example, the exponential map associated to each of these equations is non-Fredholm. This provides an important perspective on this poorly understood situation.

\textsc{Department of Mathematics, University of Colorado, Boulder CO, 80309-0395, USA}\\
Email address: \bf{pearce.washabaugh@colorado.edu}


\begin{thebibliography}{}
\bibitem{Arnold66}
Arnold, V.I., On the differential geometry of infinite-dimensional Lie groups and its application to the hydrodynamics of perfect fluids, in Vladimir I. Arnold: collected works vol. 2, Springer, New York (2014)

\bibitem{Bauer15}
Bauer, M., Kolev, B., Preston, S.C., Geometric Investigations of a Vorticity Model Equation, arXiv preprint arXiv:1504.08029v2 (2015)

\bibitem{Berger77}
Berger, M.S., Nonlinearity and functional analysis, Lectures on Nonlinear Problems in Mathematical Analysis, pg. 222, Academic, New York (1977)

\bibitem{Choi14}
 Choi, K., Hou T. Y., Kiselev, A., Luo, G., Sverak, V., and Yao, Y., On the finite-time blowup of a 1D model for
the 3D axisymmetric Euler equations, arXiv preprint arXiv:1407.4776v2 (2014)

\bibitem{Constantin94}
Constantin, P., Majda, A., Tabak, E., Formation of strong fronts in the 2D quasi-geostrophic thermal active scalar,
Nonlinearity, 7, 1495-1533 (1994) 

\bibitem{Constantin14}
Constantin, P., Vicol, V., and Wu, J., Analyticity of Lagrangian trajectories for well posed inviscid incompressible fluid models, arXiv preprint arXiv:1403.5749v2 (2014)

\bibitem{Ebin70}
Ebin, D., Marsden, J., Groups of Diffeomorphisms and the Motion of an Incompressible Fluid, Ann. of
Math. (2), 92, 102-163 (1970)

\bibitem{Ebin06}
Ebin, D., Misio\l{}ek, G., Preston, S.C., Singularities of the exponential map on the volume-preserving diffeomorphism group, GAFA, 16, 850-868 (2006)

\bibitem{Escher14}
Escher, J., Kolev, B., Right-Invariant Sobolev Metrics of Fractional Order on the Diffeomorphism Group of the Circle, Journal of Geometric Mechanics, 6, 335-372 (2014)

\bibitem{Grossman65} 
Grossman, G., Hilbert Manifolds Without Epiconjugate Points, Proc. Amer. Math. Soc., 16, 1365-1371 (1965)

\bibitem{Khesin03}
Khesin, B. and Misio\l{}ek, G., Euler equations on homogeneous spaces and Virasoro
orbits, Adv. Math., 176, 116-144 (2003)

\bibitem{Khesin13}
Khesin, B., Lennells, J., Misio\l{}ek, G., Preston, S., Curvatures of Sobolev metrics on diffeomorphism groups, PAMQ, 9, 291-332 (2013)

\bibitem{Kriegl97}
Kriegl, A., and Michor, P. W., The Convenient Setting for Global Analysis, Mathematical Surveys and Monographs, AMS,
Providence (1997)

\bibitem{Lang95}
Lang, S., Fundamentals of Differential Geometry, Vol. 191. Springer Science \& Business Media (2012)

\bibitem{Majda02}
Majda, A., Bertozzi, A., Vorticity and Incompressible Flow, Vol. 27. Cambridge University Press (2002)

\bibitem{Michor05}
Michor, P., Mumford D., Vanishing geodesic distance on spaces of submanifolds and diffeomorphisms, Doc. Math, 10, 217-245  (2005)

\bibitem{Misiolek09}
 Misio\l{}ek, G. and Preston, S., Fredholm Properties of Riemannian Exponential Maps on Diffeomorphism Groups, Invent. Math., 179, 191-227 (2009)
 
 \bibitem{Misiolek15}
 Misio\l{}ek, G., The Exponential Map Near Conjugate Points In 2D Hydrodynamics, Arnold Mathematical Journal, 1, 243-251 (2015)

\bibitem{Okamoto08}
Okamoto, H., Sakajo, T., and Wunsch, M., On a generalization of the Constantin-Lax-Majda equation, Nonlinearity 21,  2447-2461 (2008)
 
\bibitem{Preston02}
 Preston, S.C., Eulerian and Lagrangian stability of fluid motions Ph. D. Thesis, SUNY Stony Brook (2002)
 
\bibitem{Preston10}
Preston, S.C., A Geometric Rigidity Theorem for Hydrodynamical Blowup, Comm. PDE, 57, 2007-2020 (2010): 

\bibitem{Rouchon91}
Rouchon, P., Jacobi equation, Riemannian curvature and the motion of a perfect incompressible fluid, European J. Mech. B Fluids, 11, 317-336 (1992).

\bibitem{Shnirelman12}
Shnirelman, A., On the analyticity of particle trajectories in the ideal incompressible fluid, arXiv preprint arXiv:1205.5837 (2012)

\bibitem{Smale65}
Smale, S., An Infinite Dimensional Version of Sard's Theorem, Amer. J. Math. 87, 861-866 (1935)

\bibitem{Tao14}
Tao, T., Conserved quantities for the surface quasi-geostrophic equation, \\ 
https://terrytao.wordpress.com/2014/03/06/conserved-quantities-for-the-surface-quasi-geostrophic-equation/

\bibitem{Wunsch10}
Wunsch, M., The geodesic flow on the group of diffeomorphisms of the circle with fractional right invariant Sobolev metric 
Journal of Nonlinear Mathematical Physics, 17, 7-11 (2010)
\end{thebibliography}
\end{document}